\documentclass[11pt,reqno]{amsart}

\RequirePackage{amsmath}
\RequirePackage{amsopn}
\RequirePackage{amsfonts}
\RequirePackage{amsthm}
\RequirePackage{ifthen}

\usepackage{graphicx}

\theoremstyle{definition}
\newtheorem{definition}{Definition}[section]

\theoremstyle{remark}
\newtheorem{remark}{Remark}[section]

\newtheoremstyle{theorem}{0.5cm}{0.2cm}{\slshape}{ }{\bfseries}{.}{ }{}

\theoremstyle{theorem}
\newtheorem{theorem}{Theorem}[section]
\newtheorem{prop}[theorem]{Proposition}
\newtheorem{lemma}[theorem]{Lemma}
\newtheorem{corollary}[theorem]{Corollary}
\newtheorem{assumption}[theorem]{Assumption}

\newcommand{\R}{{\mathbb{R}}}

\newcommand{\E}{{\mathbb{E}}}

\newcommand{\del}{{\delta}}

\DeclareMathSymbol\square  {\mathord}{AMSa}{"03}

\renewcommand{\del}{\delta}

\newcommand{\N}{{\mathbb N}}

\newcommand{\dd}{{\rm d}}

\DeclareMathOperator{\dom}{{dom}}

\begin{document}
\title{On independence and large deviations for sublinear expectations}

    \author{Pedro Ter\' an}
	\address{Universidad de Oviedo, E-33071 Gij\'on, Spain}
	\email{teranpedro@uniovi.es}
	
		\author{Jos\'e M. Zapata}
	\address{Universidad de Murcia.  Dpto. de Estad\'{i}stica e Investigaci\'{o}n Operativa,  30100 Espinardo, Murcia, Spain}
	\email{jmzg1@um.es}
	
	\date{\today}



\date{}
\maketitle

\begin{abstract}
We prove by counterexample that a large deviation principle established by Chen and Feng [{\em Comm. Statist. Theory Methods} {\bf 45} (2016), 400--412] in the framework of sublinear expectations is incorrect. That implies that the rate function cannot, in general, be obtained by computing the Fenchel transform of the cumulant generating function, as is the case for ordinary probabilities. We derive a corrected version of that result and show that the original presentation holds under a stronger independence assumption.
\end{abstract}

\vskip 1 true cm
\noindent{\bf Keywords:} Large deviations; Cram\'er theorem; G\"artner--Ellis theorem; Negatively dependent variables;  Sublinear expectation; Upper probability.

\section{Introduction}

Large deviation principles \cite{DupEll,dembo,Puh,Cer,Var} concern the exponential rate of convergence of the probabilities in statements like, e.g., the weak law of large numbers, and their study was initiated by Cram\'er \cite{Cra}. On the other hand, starting with papers like \cite{Bad,WalFin} there arose a slow realization that the full force of Kolmogorov's axioms is not necessary for analogs of classical probabilistic limit theorems to hold.

A connection between large deviations and non-additive set functions was made when working out rigorously the analogy between large deviation principles and the weak convergence in the central limit theorem, see e.g. \cite{Puh91,O'BrVer,Aki96,Ger,Aki99}. Then large deviation results  for non-additive set functions or non-additive expectations started to appear \cite{Gul,Ned,Hu,GaoJia}. In recent years, most results have been developed in the framework of sublinear expectations \cite{Pen19}, see e.g. \cite{GaoXu,CheXio,Cao,LiuZha}. A recent abstract approach, which also covers sublinear expectations, has been developed in \cite{kupper,zapata}; we refer to \cite{Zap} for a summarized account.

Within this stream of research, Chen and Feng \cite{CheFen} and Feng \cite{Feng} considered a negative dependence property rather than independence. This is interesting since some proposed definitions of independence in the non-additive setting (e.g., \cite{Pen19}) are very stringent \cite{HuLi,CriticaPeng}.

Unfortunately, Chen and Feng's result \cite[Theorem 3.1]{CheFen} turns out to be incorrect (see Corollary \ref{ggg} below). 
Specifically, we present a counterexample that satisfies the assumptions of \cite[Theorem 3.1]{CheFen}, yet fails to satisfy the lower bound stated in that theorem. 
This implies that if a large deviation principle holds for negatively dependent variables under a sublinear expectation, the rate function cannot, in general, be obtained as the Fenchel--Legendre transform of the cumulant generating function as is the case for standard probabilities.

On the other hand, relying on the version of the G\"{a}rtner-Ellis theorem for sublinear expectations recently established in \cite{zapata}, we confirm that, although the lower bound in \cite[Theorem 3.1]{CheFen} fails, the upper bound in that theorem holds. Furthermore, we prove that the lower bound in \cite[Theorem 3.1]{CheFen} remains valid if the stronger assumption of independence, rather than negative dependence, is imposed. We also present a corrected version of Chen and Feng's theorem with an explicit calculation of the rate function.

These results have the notable implication that the relationship between limit theorems for sublinear expectations in the cases of independence and negative dependence is more intricate than in the classical setting.

The remainder of the paper is organized as follows. Section \ref{sec:preliminaries} covers the necessary preliminaries. In Section \ref{counter}, we present the counterexample announced above. In Section \ref{sec:GartnerEllis} we discuss the validity of the bounds in \cite[Theorem 3.1]{CheFen}. Section \ref{sec:conclusions} is devoted to some  concluding remarks. The proofs of some auxiliary results are presented in an appendix.

\section{Preliminaries}\label{sec:preliminaries}

Let $(\Omega,\mathcal A)$ be a measurable space. The complement of an event $A$ will be denoted by $A^c$, and its indicator function by $I_A$. The expectation against a probability measure $P$ will be denoted by $E_P$.

A {\em capacity} is a set function $v$ from $\mathcal A$ to $[0,1]$ such that $v(\O)=0$, $v(\Omega)=1$ and $v(A)\le v(B)$ whenever $A\subset B$. The {\em dual capacity} to $v$ is $\overline v\colon A\in\mathcal A\mapsto 1-v(A^c)$. The capacity {\em induced} by a random variable $X$ from $v$ (or the {\em distribution} of $X$ under $v$) is
$$v_X:A\in\mathcal B_\R\mapsto v(X\in A),$$
where $\mathcal B_\R$ denotes the Borel $\sigma$-algebra on $\mathbb{R}$.

A capacity is {\em 2-monotone} if
$$v(A\cup B)+v(A\cap B)\ge v(A)+v(B)$$
and, more generally, {\em $n$-monotone} if
$$\nu(\bigcup_{i=1}^n A_i)\ge\sum_{I\subset\{1,\ldots,n\},I\neq\O}(-1)^{|I|+1}\nu(\bigcap_{i\in I}A_i).$$
A capacity is called {\em completely monotone} (or totally monotone, or monotone of infinite order) if it is $n$-monotone for each $n\in\N$. 
It is called {\em continuous} or a {\em fuzzy measure} if $v(A_n)\to v(A)$ whenever $A_n\nearrow A$ or $A_n\searrow A$. Notice every probability distribution in $\R$ is a continuous completely monotone capacity.

The {\em Choquet integral} of a random variable $X$ against a capacity $v$ is
$$E_v[X]=\int_0^\infty v(X\ge t)\dd t-\int_{-\infty}^0[1-v(X\ge t)]\dd t,$$
provided both improper Riemann integrals exist (it can be $+\infty$ or $-\infty$ if one of the integrals is allowed to be infinite). If $P$ is a probability measure, the Choquet integral equals the Lebesgue integral and thus there is no ambiguity with the notation $E_P$.

A capacity $\mathbb V$ is called an {\em upper probability} if there exists a set of probability measures $\mathcal P$ such that $\mathbb V(A)=\sup_{P\in\mathcal P} P(A)$ for each $A\in\mathcal A$. An upper probability has an {\em upper expectation} associated as per the formula $\E[X]=\sup_{P\in\mathcal P}E_P[X]$. Observe $\mathbb V(A)=\E[I_A]$. The functional $\E$ is a {\em sublinear expectation} in the sense of Peng \cite{Pen19}, i.e., it is subadditive, positively homogeneous, translation invariant, and preserves constants.

Following \cite{CheFen}, a sequence of random variables $X_n:(\Omega,\mathcal A,P)\to\R$ will be called 
{\em negatively dependent} (with respect to $\E$) if, for any $n\in\N$ and any Borel functions $\varphi_i:\R\to [0, \infty)$,
\begin{equation}\label{eq;negDep}
\E[\varphi_1(X_1)\ldots\varphi_{n+1}(X_{n+1})]\le \E[\varphi_1(X_1)\ldots\varphi_{n}(X_{n})]\cdot \E[\varphi_{n+1}(X_{n+1})]
\end{equation}
whenever the upper expectations are finite\footnote{It is worth observing that negative dependence with respect to the expectation against a probability measure is just independence. Therefore,  this is not a genuine generalization of negative dependence of random variables in a probability space but a generalization of independence. Indeed, let $A_1,\ldots,A_{n+1}\in\mathcal A$ and also let $B_i\in\{A_i,A_i^c\}$ for each $i$. Taking $\varphi_i=I_{B_i}$ and applying negative dependence iteratively yields $P(X_1\in B_1,\ldots,X_{n+1}\in B_{n+1})\le \prod_{i=1}^{n+1}P(X_i\in B_i)$. But the sum of both sides over all possible choices of $B_i$ is $1$, which makes those inequalities identities and proves the independence of the $X_i$.}. 
If the inequalities in \eqref{eq;negDep} are identities, we say the $X_n$ are \emph{independent}.  
The $X_n$ will be called {\em identically distributed} (with respect to $\E$) if, for any $i,j\in\N$  and any Borel function $\varphi:\R\to [0, \infty)$,
$$\E[\varphi(X_i)]=\E[\varphi(X_j)]$$
whenever the upper expectations are finite.

Finally, we recall some notions from convex analysis. For a convex function $f\colon \R\to \mathbb{R}\cup\{\infty\}$, we define its \emph{effective domain} as
\[
{\rm dom}(f):=\{x\in\R\colon f(x)<\infty\}.
\]
The Fenchel--Legendre transform $f^\ast\colon \mathbb{R}\to\mathbb{R}\cup\{\infty\}$ of $f$ is defined by
\[
f^\ast(x)=\sup_{\lambda\in\R}(\lambda x - f(\lambda)).
\]
A point $y\in\R$ is said to be an \emph{exposed point} of $f^\ast$ if, for some $\lambda\in\mathbb{R}$ and all $x\neq y$,
\[
\lambda y - f^\ast(y)> \lambda x - f^\ast(x).
\]
In that case, we say that $\lambda$ is an exposed hyperplane of $x$.

\section{Counterexample}
\label{counter}

In this section, we present the announced counterexample. Let $(\Omega, \mathcal{A}, P)$ be the real line $\mathbb{R}$ equipped with its Borel $\sigma$-algebra and a probability measure $P$. Let $\mathbb{V}$ be  given by $\mathbb{V}(A) = P(A)(2 - P(A))$ for $A\in\mathcal{A}$. Since the function $x\in[0,1]\mapsto x(2-x)$ is increasing, it is easily checked that $\mathbb V$ is a capacity.

Additionally, we fix a sequence $\{X_n\}_{n \in\mathbb{N}}$ of $P$-i.i.d.~Bernoulli random variables with parameter $p \in (0,1)$. 
For each $n\in\N$,  define the sample mean $\overline{X}_n=\tfrac{1}{n}\sum_{i=1}^n X_i$.  
 For any $t\in(0,1)$, let $I_t$ denote the function
\[
I_t(x)
:=
\begin{cases}
x\ln\left(\frac{x}{t}\right)+(1-x)\ln\left(\frac{1-x}{1-t}\right), & \mbox{ if }x\in [0,1];\\
\infty,  & \mbox{ if }x\notin [0,1],
\end{cases}
\]

We have the following first result.

\begin{prop}\label{prop:LDP1}
For all closed $F\subset \mathbb{R}$ and open $G\subset\mathbb{R}$,
$$\limsup_n \frac{1}{n}\ln \mathbb V(\overline X_n\in F)\le -\inf_{x\in F}I_p(x)$$
and
$$\liminf_n \frac{1}{n}\ln \mathbb V(\overline X_n\in G)\ge -\inf_{x\in G}I_p(x).$$
\end{prop}
\begin{proof}
For all closed $F\subset \mathbb{R}$ and open $G\subset\mathbb{R}$, 
$$\limsup_n \frac{1}{n}\ln \mathbb V(\overline X_n\in F)=\limsup_n \frac{1}{n}\ln P(\overline X_n\in F)$$
and
$$\liminf_n \frac{1}{n}\ln \mathbb V(\overline X_n\in G)=\liminf_n \frac{1}{n}\ln P(\overline X_n\in G).$$
Then, the conclusion follows from Cram\'{e}r's theorem \cite[Theorem 2.2.3]{dembo}, taking into account that $I_p$ is the rate function of a Bernoulli distribution with parameter $p$, see e.g.~\cite[Exercise 2.2.23]{dembo}.  
\end{proof}

For easier readability, the proof of the following lemma is postponed to the Appendix at the end of the paper.
\begin{lemma}\label{lem:hypothesis}
The capacity $\mathbb{V}$ is an upper probability and, moreover, the sequence  $\{X_n\}_n$, the upper probability  $\mathbb{V}$ and the upper expectation $\mathbb{E}$ associated to $\mathbb{V}$  satisfy the assumptions of \cite[Theorem 3.1]{CheFen}. 
Namely, 
\begin{itemize}  
\item[(i)] $\mathbb V$ is continuous.
\item[(ii)] The sequence $\{X_n\}_n$ is negatively dependent and identically distributed with respect to $\mathbb{E}$. 
\item[(iii)] $\E[e^{\del|X_i|}]<\infty$ for all $\del>0$ and  $i\in\N.$
\item[(iv)]The mapping $\Lambda:\lambda\in\R\to[-\infty,\infty]$ given by
$$\Lambda(\lambda)=\lim_n \frac{1}{n}\sum_{i=1}^n\ln\E[e^{\lambda X_i}]$$
\noindent is well defined.
\item[(iv)]$0\in {\rm int}(\dom(\Lambda))$.
\end{itemize}
\end{lemma}
As in \cite{CheFen}, we consider the set ${\mathcal{E}}(\Lambda^\ast)$  of all exposed points of $\Lambda^\ast$ which admit an exposing hyperplane $\lambda$ such that $\lambda\in{\rm int}({\rm dom}(\Lambda))$.
 In addition, in the remainder of this section, we consider the function
\[
I(x)
:=
\begin{cases}
I_{p^2}(x), & \mbox{ if }x< p^2;\\
0,  & \mbox{ if }x\in [p^2,p(2-p)];\\
I_{p(2-p)}(x), & \mbox{ if }x> p(2-p),
\end{cases}
\]
and the set $H=(0,p^2)\cup(p^2,p(2-p))$.

We refer to the Appendix for the proof of the following lemma. 
\begin{lemma}\label{lem:conjugate} 
The following conditions hold true:
\begin{itemize}
\item[(i)] $\Lambda^\ast(x)=I(x)$ for all $x\in\mathbb{R}$.
\item[(ii)] $
\mathcal{E}(\Lambda^\ast)=H$.
\end{itemize}
\end{lemma}

After Lemmas \ref{lem:hypothesis} and \ref{lem:conjugate}, the following result is immediate.

\begin{prop}\label{prop:LDP2}
If \cite[Theorem 3.1]{CheFen} is correct then, for all closed $F\subset \mathbb{R}$ and open $G\subset \mathbb{R}$,
$$\limsup_n \frac{1}{n}\ln \mathbb V(\overline X_n\in F)\le -\inf_{x\in F}I(x)$$
and
$$\liminf_n \frac{1}{n}\ln \mathbb V(\overline X_n\in G)\ge -\inf_{x\in G\cap H}I(x).$$
\end{prop}

Finally, as a consequence of Proposition \ref{prop:LDP2} and Proposition \ref{prop:LDP1}, we obtain
\begin{corollary}\label{fin}\label{ggg}
Theorem 3.1 in \cite{CheFen} is not correct.
\end{corollary}
\begin{proof}
By contradiction, assume  that \cite[Theorem 3.1]{CheFen} holds true. Fix $0 < a < b < p^2$. By the lower bound in Proposition \ref{prop:LDP2}, we have
\begin{align*}
-\inf_{x \in (a,b) \cap H} I(x) &= -\inf_{x \in (a,b)} I_{p^2}(x) \\
                                 &= -I_{p^2}(b) \leq \liminf_{n \to \infty} \frac{1}{n} \ln \mathbb{V}(\overline{X}_n \in (a,b)),
\end{align*}
where we have used the fact that $I_{p^2}$ is decreasing on $[0, p^2]$.

On the other hand, by the upper bound in Proposition \ref{prop:LDP2},
\begin{align*}
\limsup_{n \to \infty} \frac{1}{n} \ln \mathbb{V}(\overline{X}_n \in (a,b)) &\leq \limsup_{n \to \infty} \frac{1}{n} \ln \mathbb{V}(\overline{X}_n \in [a,b]) \\
-\inf_{x \in [a,b]} I_{p^2}(x) &= -I_{p^2}(b).
\end{align*}
Combining both inequalities yields
\[
\lim_{n \to \infty} \frac{1}{n} \ln \mathbb{V}(\overline{X}_n \in (a,b)) = -I_{p^2}(b).
\]

Now a similar argument, using the bounds provided in Proposition~\ref{prop:LDP1}, gives
\[
\lim_{n \to \infty} \frac{1}{n} \ln \mathbb{V}(\overline{X}_n \in (a,b)) = -I_{p}(b).
\]

Since \(b\) was arbitrary, we conclude that \(I_{p^2}(b) = I_{p}(b)\) for all \(b \in (0, p^2)\). However, this contradicts the trivial fact that \(I_{p^2} \neq I_{p}\) on \((0, p^2)\) (see Figure~\ref{pic}).  
By {\it reductio ad absurdum}, we obtain the desired conclusion.
\end{proof}

\begin{figure}[htb]\label{pic}
\centering
\includegraphics[width=10cm]{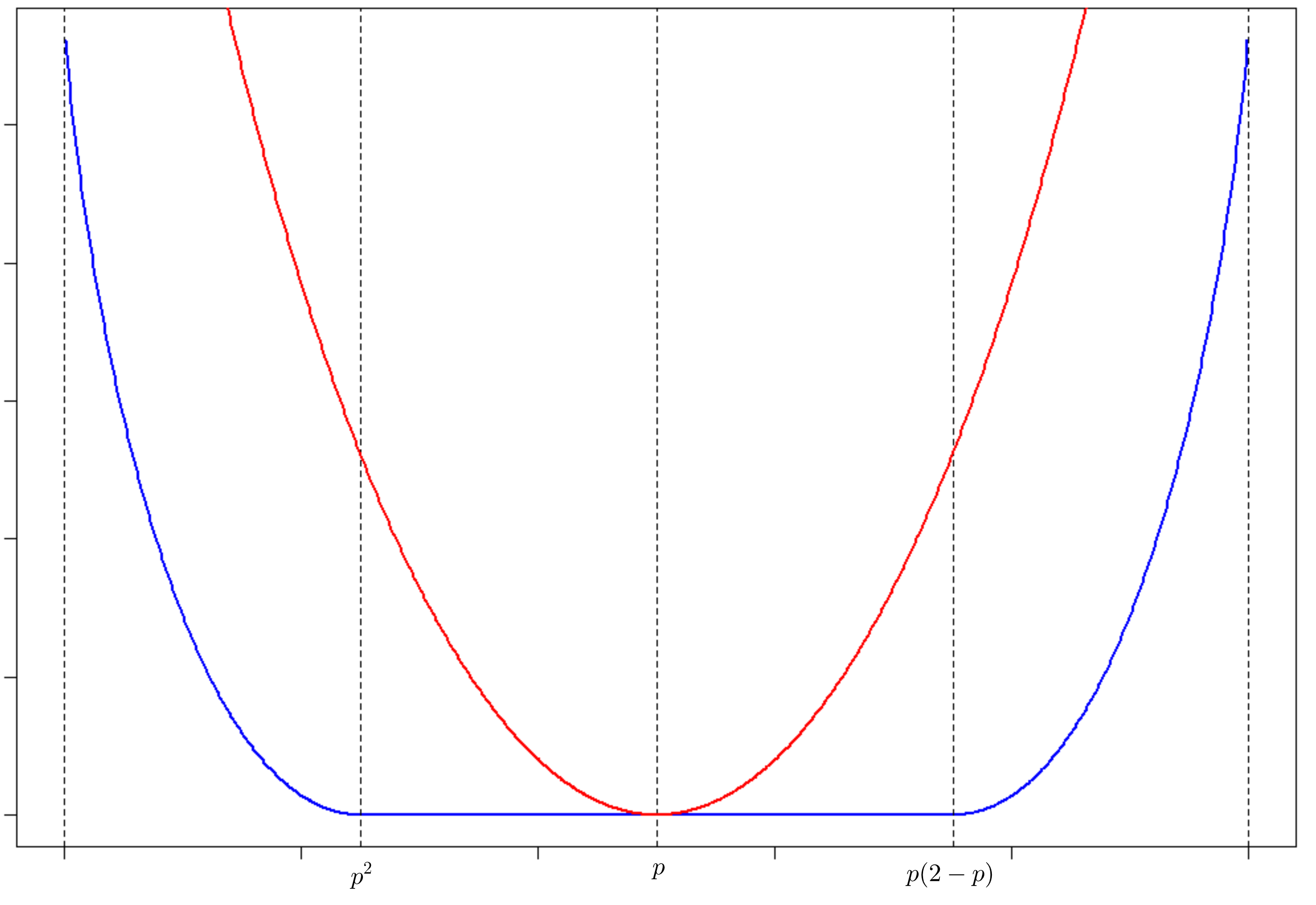}
\caption{The function $I_p$ is plotted in red, and the function $I$ is plotted in blue.}
\end{figure}

\begin{remark}\label{rem:bounds}
Given a closed set $F \subset \R$, applying the upper bound from Proposition \ref{prop:LDP1}, we obtain
$$
\limsup_{n \to \infty} \frac{1}{n} \ln \mathbb{V}(\overline{X}_n \in F) \leq -\inf_{x \in F} I_p(x) \leq -\inf_{x \in F} I(x),
$$
where the second inequality follows from the fact that $I \leq I_p$. Thus, our example satisfies the upper bound in \cite[Theorem 3.1]{CheFen}. 
However, by Corollary \ref{fin}, the lower bound in \cite[Theorem 3.1]{CheFen} must fail.  In fact, what the proof of Corollary \ref{fin} shows is that the lower bound in \cite[Theorem 3.1]{CheFen} leads to a contradiction with the upper bound in Cram\'{e}r's theorem.
\end{remark}

\section{Large deviation principle for negatively dependent random variables}
\label{sec:GartnerEllis}

In this section, we provide an alternative version of the large deviation principle stated in \cite[Theorem 3.1]{CheFen}, which we have shown to be incorrect in the previous section. We rely on a general version of the G\"{a}rtner-Ellis theorem, recently provided in \cite[Theorem 5.4]{zapata}, which establishes large deviation bounds for general capacities on regular topological spaces.

Let $(\Omega, \mathcal{A}, P)$ be a probability space. We fix an upper probability 
\[
\mathbb{V}(A) := \sup_{P \in \mathcal{P}} P(A) \quad \text{for } A \in \mathcal{A},
\]
and denote by $\mathbb{E}$ the associated upper expectation.  
Additionally, consider a sequence $\{X_n\}_{n\in\N}$ of random variables on $(\Omega, \mathcal{A})$, and for each $n \in \mathbb{N}$, define the corresponding sample mean $\overline{X}_n := \frac{1}{n} \sum_{i=1}^n X_i$.  
Note that we do not impose (for now) any independence assumptions on $\{X_n\}_{n\in\N}$.

The notion of an exponentially tight sequence is widely used in large deviation theory. It extends naturally to capacities as follows.

\begin{definition}
The sequence of random variables $\{\overline{X}_n\}_{n\in\N}$ is said to be \emph{exponentially tight} if, for every $N \in \mathbb{N}$, there exists a compact set $K \subset \mathbb{R}$ such that
\[
\limsup_{n \to \infty} \frac{1}{n} \ln \mathbb{V}(\overline{X}_n \in K^c) \leq -N.
\]
\end{definition}

In the following, we define the function ${\Gamma}\colon \mathbb{R}\to \mathbb{R}\cup\{\infty\}$ by 
\[
{\Gamma}(\lambda):=\limsup_{n\to\infty}\frac{1}{n}\ln\mathbb{E}[e^{n\lambda \overline{X}_n}].
\]
We define $\tilde{\mathcal{E}}(\Gamma^\ast)$ as the set of all exposed points $x$ of $\Gamma^\ast$ for which there exists an exposing hyperplane $\lambda$ such that the limit superior in $\Gamma(\lambda)$ is actually a limit, and  $\Gamma(t\lambda) < \infty$ for some $t > 1$.


 As a consequence of \cite[Theorem 5.4]{zapata}, we obtain the following result.\footnote{In order to apply  \cite[Theorem 5.4]{zapata}, we consider $\mathcal{E}_n(f):=\mathbb{E}[e^{n f(\overline{X}_n)}]$, $\mu_n(A):=\mathbb{V}(X_n\in A)$, and $\mathcal{H}:=\{f_\lambda\colon \lambda\in\mathbb{R}\}$ where $f_\lambda(x)=\lambda x$  for $x\in \mathbb{R}$. }

\begin{theorem}\label{thm:gartnerEllis}
Suppose that the sequence $\{\overline{X_n}\}_{n\in\mathbb{N}}
$ is exponentially tight. Then, for all closed $F\subset\mathbb{R}$ and open $G\subset \R$, 
\[
\limsup_{n\to\infty}\frac{1}{n}\ln\mathbb{V}(\overline{X}_n\in F)\le -\inf_{x\in F}{\Gamma}^\ast(x)
\]
and
\[
\liminf_{n\to\infty}\frac{1}{n}\ln\mathbb{V}(\overline{X}_n\in G)\ge -\inf_{x\in G\cap \tilde{\mathcal{E}}({\Gamma}^\ast)}{\Gamma}^\ast(x).
\]
\end{theorem}

In the following, we adopt the same assumptions as those employed in the proof of Theorem 3.1 in \cite{CheFen}. We aim to study which conclusions can really be obtained from Theorem \ref{thm:gartnerEllis} under such   conditions.

\begin{assumption}\cite[Assumption 3.1]{CheFen}\label{ass:moments}
For every $\lambda\in\mathbb{R}$, the limit 
\[
{\Lambda}(\lambda):=\lim_{n\to\infty}\frac{1}{n}\sum_{i=1}^n \ln \mathbb{E}[e^{\lambda X_i}]
\]
exists as an extended real number. Furthermore, $0\in{\rm int}(\dom(\Lambda))$.
\end{assumption}

\begin{remark}\label{rem:independence}
We can establish some relations between the functions $\Lambda$ and $\Gamma$. If the random variables $\{X_n\}_{n\in\N}$ are negatively dependent then 
\[
\Gamma(\lambda) \leq \Lambda(\lambda) \quad \text{for all } \lambda \in \mathbb{R}.
\]
As a consequence, we have
\[
\Gamma^\ast(x) \geq \Lambda^\ast(x) \quad \text{for all } x \in \mathbb{R}.
\]
Moreover, both inequalities become equalities if $\{X_n\}_{n\in\N}$ are independent.

\end{remark}

Recall $\mathcal{E}(\Lambda^\ast)$ is the set of all exposed  points of $\Lambda^\ast$ with a exposing hyperplane $\lambda$ such that $\lambda\in{\rm int}({\rm dom}(\Lambda))$.
 Theorem \ref{thm:gartnerEllis} implies the following. 

\begin{corollary}\label{cor:gartnerEllis}
Suppose that $\{X_n\}_{n\in\N}$ are negatively dependent and that $\mathbb{E}[e^{\delta|X_n|}]<\infty$ for all $\delta>0$ and $n\in\N$. Let Assumption \ref{ass:moments} hold. Then, for all closed  $F\subset\mathbb{R}$ and open $G\subset\mathbb{R}$, 
\begin{equation}\label{eq:UpperGartnerEllis}
\limsup_{n\to\infty}\frac{1}{n}\ln\mathbb{V}(\overline{X}_n\in F)\le -\inf_{x\in F}{\Gamma}^\ast(x)\le -\inf_{x\in F}{\Lambda}^\ast(x)
\end{equation}
and
\begin{equation}\label{eq:LowerGartnerEllis}
\liminf_{n\to\infty}\frac{1}{n}\ln\mathbb{V}(\overline{X}_n\in G)\ge -\inf_{x\in G\cap \tilde{\mathcal{E}}({\Gamma}^\ast)}{\Gamma}^\ast(x).
\end{equation}
If, moreover, $\{X_n\}_{n\in\N}$ are independent, then ${\Lambda}={\Gamma}$, ${\Lambda}^\ast={\Gamma}^\ast$, ${\mathcal{E}}({\Lambda}^\ast)\subset\tilde{\mathcal{E}}({\Gamma}^\ast)$, and the conclusions of \cite[Theorem 3.1]{CheFen} hold true.
\end{corollary}
\begin{proof}
We first prove that the sequence $\{\overline{X}_n\}_{n}$ is exponentially tight.  
Fixed $N\in\mathbb{N}$, consider the compact set $K_N:=[-N,N]$. 
Then,
\[
\mathbb{V}(\overline{X}_n\in K_N^c)\le \mathbb{V}(\overline{X}_n<-N) + \mathbb{V}(N<\overline{X}_n).
\]
Since $0$ is in the interior of $\dom{{\Lambda}}$, we can find $\varepsilon>0$ such that ${\Lambda}(\varepsilon)<\infty$ and ${\Lambda}(-\varepsilon)<\infty$. By the Markov inequality and negative dependence,
\[
\mathbb{V}(\overline{X}_n>N)\le \frac{\mathbb{E}[e^{n\varepsilon\overline{X}_n}]}{e^{n\varepsilon N}}\le \frac{\prod_{i=1}^n\mathbb{E}[e^{\varepsilon {X}_i}]}{e^{n\varepsilon N}}.
\]
From there, we get
\[
\frac{1}{n}\ln \mathbb{V}(\overline{X}_n>N)\le -\varepsilon N + \frac{1}{n}\sum_{i=1}^n\ln \mathbb{E}[e^{\varepsilon X_i}].
\]
Taking the limit superior as $n\to \infty$ results in
\[
\limsup_{n\to\infty}\frac{1}{n}\ln \mathbb{V}(\overline{X}_n>N)\le -\varepsilon N + {\Lambda}(\varepsilon).
\]
A similar argument leads to
\[
\limsup_{n\to\infty}\frac{1}{n}\ln \mathbb{V}(\overline{X}_n<-N)\le -\varepsilon N + {\Lambda}(-\varepsilon).
\]

Then,
\begin{align*}
\limsup_{n\to\infty}\frac{1}{n}\ln\mathbb{V}(\overline{X}_n\in K_N^c)&\le 
\limsup_{n\to\infty}\frac{1}{n}\ln\Big( \mathbb{V}(\overline{X}_n<-N)+\mathbb{V}(N<\overline{X}_n)\Big)\\
&\le 
\limsup_{n\to\infty}\frac{1}{n}\ln\Big( 2\max\left\{\mathbb{V}(\overline{X}_n<-N),\mathbb{V}(N<\overline{X}_n)\right\}\Big)\\
&=
\limsup_{n\to\infty}\frac{1}{n}\ln\Big(\max\left\{\mathbb{V}(\overline{X}_n<-N),\mathbb{V}(N<\overline{X}_n)\right\}\Big)\\
&= \max\left\{ \limsup_{n\to\infty}\frac{1}{n}\ln\mathbb{V}(\overline{X}_n<-N) , \limsup_{n\to\infty}\frac{1}{n}\ln\mathbb{V}(N<\overline{X}_n)\right\}\\
&\le-\varepsilon N + \max\left\{{\Lambda}(-\varepsilon)  ,{\Lambda}(\varepsilon) \right\}.
\end{align*}
Since the right-hand side goes to $-\infty$ as $N\to\infty$, the sequence $\{\overline{X}_n\}_{n\in\N}$ is  exponentially tight. 
Then,  inequalities \eqref{eq:UpperGartnerEllis} and \eqref{eq:LowerGartnerEllis} follow from Theorem \ref{thm:gartnerEllis} taking into account that $\Gamma^\ast(x) \geq \Lambda^\ast(x)$ (see Remark~\ref{rem:independence}).

If the random variables $\{X_n\}_{n\in\N}$ are independent, then $\Gamma^\ast=\Lambda^\ast$ as observed in Remark~\ref{rem:independence}. In that case, we have that $\mathcal{E}(\Lambda^\ast)=\mathcal{E}(\Gamma^\ast)\subset \tilde{\mathcal{E}}(\Gamma^\ast)$. Then,  it follows from \eqref{eq:LowerGartnerEllis} that 
\[
\liminf_{n\to\infty}\frac{1}{n}\ln\mathbb{V}(\overline{X}_n\in G)\ge -\inf_{x\in G\cap \tilde{\mathcal{E}}({\Gamma}^\ast)}{\Gamma}^\ast(x)\ge 
-\inf_{x\in G\cap {\mathcal{E}}({\Gamma}^\ast)}{\Lambda}^\ast(x),
\]
obtaining the lower bound in  \cite[Theorem 3.1]{CheFen}. 
Finally, the upper bound in \cite[Theorem 3.1]{CheFen} is immediate from  \eqref{eq:UpperGartnerEllis}.
\end{proof}

\begin{remark}
The inequality \eqref{eq:UpperGartnerEllis} in Corollary~\ref{cor:gartnerEllis} shows that the upper bound in \cite[Theorem 3.1]{CheFen} is correct; however, the lower bound cannot be right due to the counterexample provided in the previous section. This is consistent with what we already observed in Remark \ref{rem:bounds}.

On the other hand, 
as shown in Corollary~\ref{cor:gartnerEllis}, if we assume the stronger condition of independence, then the lower bound in \cite[Theorem 3.1]{CheFen} holds  and the conclusions of \cite[Theorem 3.1]{CheFen} are valid. 
\end{remark}

\section{Concluding remarks}\label{sec:conclusions}

Our example relies on finding an upper probability and its nonlinear expectation, $\mathbb V$ and $\E$, such that the asymptotic behaviour of variables in that space is directly linked to  ordinary probability measures and their expectation functionals. Observe that almost sure convergence
$$P(|\overline X_n-E_P[X_1]|\not\to 0)=0$$
is equivalent, since $\mathbb V=P(2-P)$, to
$$\mathbb V(|\overline X_n-E_P[X_1]|\not\to 0)=0,$$
so $\overline X_n$ does not converge $\mathbb V$-almost surely to $\E[X_1]$ but to $E_P[X_1]$ (see \cite{Chareka} for a related argument in the case of the weak law of large numbers). 
Thus, we show that a naive `word-by-word' extension of the large deviation principle to negatively dependent random variables creates a contradiction with the probabilistic form of this theorem. As shown, that does not preclude negatively dependent sequences from satisfying a large deviation principle under an upper expectation. It only establishes that the rate function cannot be calculated as the Fenchel--Legendre transform of the function $\Lambda(\lambda)=\ln\E[e^{\lambda X}]$. 


The assumption that $\varphi_i$ are Borel functions in the definition of negative dependence and independence can be weakened as the proof in Section 4 only uses functions of the form $\varphi_i(x)=e^{{\delta}x}$.

\begin{appendix}

\section{Proofs of some auxiliary results}
\label{proofs}

In this section we provide the proofs of Lemma \ref{lem:hypothesis} and Lemma \ref{lem:conjugate}. 
As in Section \ref{counter}, let $(\Omega, \mathcal{A}, P)$ be the real line $\mathbb{R}$ equipped with its Borel $\sigma$-algebra and a probability measure $P$. Define $\mathbb{V}$ as the capacity given by $\mathbb{V}(A) = P(A)(2 - P(A))$. Additionally, we fix a sequence $\{X_n\}_{n \in\N}$ of $P$-i.i.d.~Bernoulli random variables with parameter $p \in (0,1)$ and, for each $n$, define the sample mean $\overline{X}_n=\tfrac{1}{n}\sum_{i=1}^n X_n$.  

In order to prove Lemma~\ref{lem:hypothesis}, we  use several results from \cite{TAMS}. Since the language for capacities in that paper differs from the definitions in this one, the reader is referred to \cite[Section 1]{TAMS} for more detailed information. For our purposes, it suffices to know that `capacity$\Rightarrow$topological capacity' in the terminology of that paper. 

Recall that $E_{\mathbb V}$ denotes the Choquet integral against the capacity $\mathbb{V}$. We have the following.
\begin{lemma}\label{lemoso}
The capacity $\mathbb V$ is an upper probability. Moreover, its associated upper expectation is $\E=E_{\mathbb V}$.
\end{lemma}

\begin{proof}
Let $v$ denote the dual capacity of $\mathbb V$. For any $A\in\mathcal B_\R$,
$$v(A)=1-\mathbb V(A^c)=1-P(A^c)(2-P(A^c))=(1-P(A^c))^2=P(A)^2.$$
By \cite[Proposition 5.6.(a)]{TAMS}, $v$ is a completely monotone capacity in the sense of \cite{TAMS}, in particular a 2-monotone topological capacity in the sense of that paper. Hence \cite[Proposition 3.5]{TAMS} applies: for any random variable $Y$ such that $E_v[Y]$ is finite,
$$E_v[Y]=\inf_{Q\ge v}E_Q[Y],$$
where $Q$ ranges over all probability measures such that $Q(A)\ge v(A)$ for all $A\in\mathcal B_\R$. By \cite[Lemma 3.4]{TAMS},
$$E_{\mathbb V}[Y]=-E_v[-Y]=\sup_{Q\ge v}(-E_Q[-Y])=\sup_{Q\ge v}E_Q[Y]$$
so $E_{\mathbb V}$ is an upper expectation. Applying the same formula to an arbitrary indicator function $I_A$ we have
$$\mathbb V(A)=E_{\mathbb V}[I_A]=\sup_{Q\ge v}E_Q[I_A]=\sup_{Q\ge v}Q(A)$$
so indeed $\mathbb V$ is an upper probability and $\E=E_{\mathbb V}$.
\end{proof}

In the sequel, $\E$ will always denote the upper expectation associated to $\mathbb V$.

\begin{lemma}\label{max}
Let $Y$ be a random variable such that $\E[Y]$ is finite. Then, for any random variables $Z,Z'$ which are $P$-i.i.d.~as $Y$, the variable $\max\{Z,Z'\}$ is $P$-integrable and
$$\E[Y]=E_P[\max\{Z,Z'\}].$$
\end{lemma}
\begin{proof}
As shown in the previous proof, $\mathbb V$ is the dual capacity of $P^2$. From \cite[Proposition 5.6.(c)]{TAMS} (where, like in this paper, the notation $\overline{P^2}$ represents the dual capacity of $P^2$),
$$E_{\mathbb V}[Y]=E_P[\max\{Z,Z'\}].$$
But, by Lemma \ref{lemoso}, $\E[Y]=E_{\mathbb V}[Y]$.
\end{proof}

A crucial element of our argument is as follows.

\begin{prop}\label{negdep}
Let $\{Y_n\}_{n\in\N}$ be a sequence of random variables on $(\Omega,\mathcal{A},P)$.
If $\{Y_n\}_{n\in\N}$ are i.i.d.~with respect to $P$, then $\{Y_n\}_{n\in\N}$ are negatively dependent and identically distributed with respect to $\mathbb E$.
\end{prop}
\begin{proof}
Fix $n\in\N$. Let $\varphi_1,\ldots,\varphi_{n+1}:\R\to[0,\infty)$ be measurable functions.
Let us prove the negative dependence of $Y_n$ with respect to $\E$. By Lemma \ref{max},
$$\E[\varphi_1(Y_1)\ldots\varphi_n(Y_n)]=E_P[\max\{Z,Z'\}]$$
where $Z,Z'$ are $P$-independent random variables identically distributed as $\varphi_1(Y_1)\ldots\varphi_n(Y_n)$. Therefore
$$\E[\varphi_1(Y_1)\ldots\varphi_n(Y_n)]=E_P[\max\{\varphi_1(Y_1)\ldots\varphi_n(Y_n),\varphi_1(Y_{n+2})\ldots\varphi_n(Y_{2n+1})\}].$$
Analogously,
$$\E[\varphi_{n+1}(Y_{n+1})]=E_P[\max\{\varphi_{n+1}(Y_{n+1}),\varphi_{n+1}(Y_{2n+2})\}]$$
and then, by the $P$-independence of all $2n+2$ variables,
$$\E[\varphi_1(Y_1)\ldots\varphi_n(Y_n)]\cdot \E[\varphi_{n+1}(Y_{n+1})]$$
$$=E_P[\max\{\varphi_1(Y_1)\ldots\varphi_n(Y_n),\varphi_1(Y_{n+2})\ldots\varphi_n(Y_{2n+1})\}\cdot \max\{\varphi_{n+1}(Y_{n+1}),\varphi_{n+1}(Y_{2n+2})\}]$$
$$\ge E_P[\max\{\varphi_1(Y_1)\ldots\varphi_n(Y_n)\cdot \varphi_{n+1}(Y_{n+1}), \varphi_1(Y_{n+2})\ldots\varphi_n(Y_{2n+1})\cdot \varphi_{n+1}(Y_{2n+2})\}]$$
$$=\E[\varphi_1(Y_1)\ldots\varphi_n(Y_n)\cdot \varphi_{n+1}(Y_{n+1})].$$

The proof of identical distribution is similar. Let $i\neq j$ and let $\varphi:\R\to\R$ be measurable. Since all $Y_n$ are independent and identically distributed with respect to $P$, all $(Y_n,Y_{n+1})$ are $P$-i.i.d.~random vectors, and then all $\max\{\varphi(Y_n),\varphi(Y_{n+1})\}$ are $P$-i.i.d.~random variables as well. Accordingly,
$$\E[\varphi(Y_i)]=E_P[\max\{\varphi(Y_i),\varphi(Y_{j})\}]=\E[\varphi(Y_j)].$$
The proof is complete.
\end{proof}

We are ready to prove Lemma \ref{lem:hypothesis}.

\begin{proof}

Lemma \ref{lemoso} proved that $\mathbb V$ is an upper probability.

{\it Part (i).} Continuity is obvious from that of $P$. 

{\it Part (ii).} The sequence $\{X_n\}_n$ is negatively dependent and identically distributed with respect to $\mathbb{E}$ due to Proposition \ref{negdep}.

{\it Part (iii).} Since $0\le X_i\le 1$,
$$\E[e^{\del|X_i|}]\le e^\del<\infty.$$

{\it Part (iv).} For any fixed $\lambda\in\R$, by Lemma \ref{max},
\begin{align*}
\E[e^{\lambda X_i}]&=E_P[\max\{e^{\lambda X_i},e^{\lambda X_{i+1}}\}]\\
&=(1-p)^2+ 2p(1-p)\max\{1,e^{\lambda}\}+ p^2 e^\lambda.
\end{align*}
Therefore
\begin{align*}
\Lambda(\lambda)&=\lim_n \frac{1}{n}\sum_{i=1}^n\ln\E[e^{\lambda X_i}]\\
&=
\begin{cases}
\ln\left(1-p^2+ p^2 e^\lambda\right), & \mbox{ if }\lambda< 0;\\
\ln\left((1-p)^2+p(2-p)e^\lambda\right), & \mbox{ if }\lambda\ge 0.
\end{cases}
\end{align*}
In particular, $\Lambda(\lambda)=\lim_n \tfrac{1}{n}\sum_{i=1}^n\ln\E[e^{\lambda X_i}]$
is well defined.

{\it Part (v).} This is trivial since $\dom(\Lambda)=\R$.
\end{proof}

We now turn to the proof of Lemma~\ref{lem:conjugate}. 

\begin{proof}
{\it Part (i).} 
We first notice 
\[
\Lambda(\lambda)=
\begin{cases}
\Lambda_{p^2}(\lambda), & \mbox{ if }\lambda< 0;\\
\Lambda_{p(2-p)}(\lambda), & \mbox{ if }\lambda\ge 0,
\end{cases}
\]
where $\Lambda_t(\lambda):=\ln(1-t + t e^\lambda)$ is the cumulant generating function of a  Bernoulli distribution with parameter $t\in(0,1)$.  Applying \cite[Lemma 2.2.5]{dembo}, the rate function $I_t$ of this distribution verifies
\begin{equation}\label{eq:FenchelBern}
I_t(x)=\Lambda^\ast_t(x)
=
\begin{cases}
\sup_{\lambda\ge 0}(x\lambda - \Lambda_t(\lambda)), & \mbox{ if }x\ge t;\\
\sup_{\lambda< 0}(x\lambda - \Lambda_t(\lambda)), & \mbox{ if }x< t.
\end{cases}
\end{equation}
Fixed $x\in\mathbb{R}$, consider the function $f\colon\mathbb{R}\to [-\infty,\infty)$ defined by
\[
f(\lambda):=x\lambda - \Lambda(\lambda)=
\begin{cases}
f_{p^2}(\lambda), & \mbox{ if }\lambda< 0;\\
f_{p(2-p)}(\lambda), & \mbox{ if }\lambda\ge  0,
\end{cases}
\]
where $f_t(\lambda):=x\lambda - \Lambda_t(\lambda)=x\lambda - \ln(1-t + t e^\lambda)$. 
We have 
\begin{equation}\label{eq:supremum}
\Lambda^\ast(x)=\sup_{\lambda\in\mathbb{R}} f(\lambda)=\max\left\{\sup_{\lambda\le 0} f_{p^2}(\lambda)  ,\sup_{\lambda\ge 0} f_{p(2-p)}(\lambda) \right\}.
\end{equation}

In general, for any given $t$, the function $f_t$ is  increasing on $(-\infty,0]$ if $x\ge t$, and $f_t$ is  decreasing on $[0,\infty)$ if $x\le t$. 
Taking into account this we discuss the value of \eqref{eq:supremum}, distinguishing three cases:

{\em Case I.} Assume $x\in [p^2,p(2-p)]$. We have that $f_{p^2}$ is  increasing on $(-\infty,0]$, and  $f_{p(2-p)}$ is  decreasing on $[0,\infty)$. Thus, due to \eqref{eq:supremum}, we get $\Lambda^\ast(x)=f(0)=0$.

{\em Case II.} Assume $x>p(2-p)$. Since, in particular, $x\ge p^2$, we have that $f_{p^2}$ is increasing on $(-\infty,0]$. Thus, due to \eqref{eq:supremum},  
\[
\Lambda^\ast(x)=\sup_{\lambda\ge 0} f_{p(2-p)}(\lambda)=\Lambda_{p(2-p)}^\ast(x)=I_{p(2-p)}(x), 
\]
where the last equality follows from \eqref{eq:FenchelBern}.

{\em Case III.} Assume $x<p^2$. Since, in particular, $x\le p(2-p)$, we have that $f_{p(2-p)}$ is  decreasing on $[0,\infty)$. Thus, due to \eqref{eq:supremum}, 
\[
\Lambda^\ast(x)=\sup_{\lambda\le 0} f_{p^2}(\lambda)=\Lambda_{p^2}^\ast(x)=I_{p^2}(x), 
\]
where the last equality follows from \eqref{eq:FenchelBern}. 

In summary, we obtain
\[
\Lambda^\ast(x)=
\begin{cases}
I_{p^2}(x), & \mbox{ if }x< p^2;\\
0,  & \mbox{ if }x\in [p^2,p(2-p)];\\
I_{p(2-p)}(x), & \mbox{ if }x> p(2-p).
\end{cases}
\]
This completes the proof of (i) in Lemma~\ref{lem:conjugate}. 

{\it Part (ii).} We need to prove that  $$\mathcal{E}(\Lambda^\ast)=(0,p^2)\cup\Big(p(2-p),1\Big).$$ 
 A point $y\in (0,1)$ is exposed if and only if there exists  $\lambda\in\mathbb{R}$ such the function $g_{\lambda}:=\lambda x - I(x)$ attains a maximum at $y$ which  is not attained at any other point. 
 For every $\lambda\in\mathbb{R}$, the function $g_{\lambda}$ is differentiable on $(0,1)$. Thus, $y\in(0,1)$ is  exposed  if and only if there exists $\lambda$ such that $y$ is the unique point such that $g^\prime_\lambda(y)=0$ and  $g^{\prime\prime}_\lambda(y)<0$. 

Let $y\in (0,p^2)$. 
Since
$$g^\prime_\lambda(x)=\lambda-\ln\frac{x(1-p^2)}{(1-x)p^2}\mbox{ for }x\in (0,p^2),$$
for the value
$$\lambda=\ln\frac{y(1-p^2)}{(1-y)p^2}$$
the function $g_\lambda$ has a critical point at $y$. Furthermore, this critical point is unique because $g^\prime_\lambda$ has no other zeros besides $y$. Since
$$g^{\prime\prime}(y)=\frac{-1}{y(1-y)}<0,$$
indeed $y$ is a maximum. We conclude  $(0,p^2)\subset \mathcal{E}(\Lambda^\ast)$.

A similar argument shows that $(p(2-p),1)\subset \mathcal{E}(\Lambda^\ast)$. 

Let $y\in [p^2,p(2-p)]$. Since 
$$g^\prime_\lambda(x)=\lambda\mbox{ for }x\in [p^2,p(2-p)],$$
we have that, if $y$ is a critical point, then $\lambda=0$. In that case, $g_\lambda=0$ on $ [p^2,p(2-p)]$, and $y$ is not the unique critical point because, in fact, all the points in $ [p^2,p(2-p)]$ are critical. We conclude that $[p^2,p(2-p)]$ has no exposed points. 

Finally, $\mathbb{R}\setminus(0,1)$ has no exposed points. By contradiction, assume that $y\in \mathbb{R}\setminus(0,1)$ is  exposed. Taking any exposing hyperplane $\lambda$ and any $x\in\mathbb{R}$ we have 
\[
-\infty\le \lambda x - \Lambda^\ast(x) < \lambda x -\Lambda^\ast(y)=\lambda x - \infty=-\infty.
\]
We get $-\infty<-\infty$, reaching  a contradiction. This completes the proof of (ii) in Lemma~\ref{lem:conjugate}.
\end{proof}

\end{appendix}

\vskip 2 true cm

\noindent{\small
{\bf Acknowledgements. }
The second author was partially supported by Fundaci\'{o}n S\'{e}neca - ACyT Regi\'{o}n de Murcia project
21955/PI/22, and by Grant PID2022-137396NB-I00 funded by MICIU/AEI/10.13039/501100011033 and ``ERDF A way of making Europe''.
}


\begin{thebibliography}{29}


\bibitem{Aki96} M. Akian (1996). On the continuity of the Cram\'er transform. Research report, INRIA, Rocquencourt.



\bibitem{Aki99} M. Akian (1999). Densities of idempotent measures and large deviations. Trans. Amer. Math. Soc. 351, 4515--4543.


\bibitem{Bad} R. Badard (1982). The law of large numbers for fuzzy processes and the estimation problem. Inform. Sci. 28, 161-178.





\bibitem{Cao} X. Cao (2014). An upper bound or large deviations for capacities. Math. Problems Eng. 2014, 516291, 1--6.

\bibitem{Cer} R. Cerf (2007). On Cram\'er's theory in infinite dimensions. Soc. Math. France, Paris.


\bibitem{CheFen} Z. Chen, X. Feng (2016). Large deviation for negatively dependent random variables under sublinear expectations. Comm. Statist. Theory Methods 45, 400--412.





\bibitem{CheXio} Z. J. Chen, J. Xiong (2012). Large deviation principle for diffusion processes under a sublinear expectation. Sci. China Math. 55, 2205--2216.



\bibitem{Cra} H. Cram\'er (1938). Sur un nouveau th\'eor\`eme--limite de la th\'eorie des probabilit\'es. Actualit\'es Sci. Indust. 736, 5--23.











\bibitem{dembo}
A. Dembo, O. Zeitouni (1998). Large deviations techniques and applications. 2nd edition, Springer, New York.









\bibitem{DupEll} P. Dupuis, R. S. Ellis (1997). A weak convergence approach to the theory of large deviations. Wiley, New York.








\bibitem{Feng} X. Feng (2017). Self-normalized large deviations under sublinear expectation. Statist. Probab. Letters 123, 77--83.



\bibitem{GaoJia} F. Gao, H. Jiang (2010). Large deviations for stochastic differential equations driven by $G$-Brownian motion. Stoch. Proc. Appl. 120, 2212--2240.

\bibitem{GaoXu} F. Gao, M. Xu (2012). Relative entropy and large deviation under sublinear expectations. Acta Math. Sci. 32, 1826--1834.

\bibitem{Ger} B. Gerritse (1996). Varadhan's theorem for capacities. Comm. Math. Univ. Carolinae 37, 667--690.

\bibitem{Gul} O. V. Gulinsky (2003). The principle of the largest terms and quantum large deviations. Kybernetika 39, 229--247.



\bibitem{Hu} F. Hu (2010). On Cram\'er's theorem for capacities. C. R. Acad. Sci. Paris Ser. I 348, 1009--1013.


\bibitem{HuLi} M. Hu, X. Li (2014). Independence under the $G$-expectation framework. J. Theoret. Probab. 27, 1011--1020.

\bibitem{kupper} M.~Kupper, Jos\'{e} M. Zapata (2023). Weakly maxitive set functions and their possibility distributions. Fuzzy Sets Syst. 467, 108506.





\bibitem{LiuZha} W. Liu, Y. Zhang (2023). Large deviation principle for linear processes generated by real stationary sequences under the sub-linear expectation. Comm. Statist. Theory Methods 52, 5727--5741.




\bibitem{Ned} L. Nedovi\'c, N. M. Ralevi\'c, T. Grbi\'c (2005). Large deviation principles with generated pseudo measures. Fuzzy Sets and Systems 155, 65--76.


\bibitem{O'BrVer} G. L. O'Brien, W. Verwaat (1991). Capacities, large deviations, and loglog laws. In: Stable processes and related topics (S. Cambanis, G. Samoridnitsky, M. S. Taqqu, eds.), 43--83. Birkh\"auser, Boston.



\bibitem{Pen19} S. Peng (2019). Nonlinear expectations and stochastic calculus under uncertainty. Springer, Berlin.



\bibitem{Puh91} A. Puhalskii (1991). On functional principle of large deviation. In: V. Sazonov, T. Shervashidze (eds.), New trends in probability and statistics, Vol. 1, 198--219. VSP / Moks'las, Utrecht / Vilnius.

\bibitem{Puh} A. Puhalskii (2001). Large deviations and idempotent probability. Chapman \& Hall/CRC, Boca Raton.
















\bibitem{TAMS}
P. Ter\'an (2014). Laws of large numbers without additivity. Trans. Amer. Math. Soc. 366, 5431--5451.

\bibitem{Chareka}
P. Ter\'an (2015). Counterexamples to a Central Limit Theorem and a Weak Law of Large Numbers for capacities. Statist. Probab. Lett. 96, 185--189.



\bibitem{CriticaPeng}
P. Ter\'an. On independence of random variables under sublinear expectations. Int. J. Gen. Systems, accepted for publication.

\bibitem{Var}
S. R. S. Varadhan (2016). Large deviations. Amer. Math. Soc., Providence.




\bibitem{WalFin} P. Walley, T. L. Fine (1982). Towards a frequentist theory of upper and lower probability. Ann. Statist. 10, 741--761.







\bibitem{zapata}
J.~M. Zapata (2023). Representation of weakly maxitive monetary risk measures and their rate functions. J. Math. Anal. Appl., 524, 127072, 1--20.

\bibitem{Zap}
J. M. Zapata (2024). Large deviation theory without probability measures. In: J. Ansari {\em et al.} (eds.) Combining, modelling and analyzing imprecision, randomness and dependence, 554--561. Springer, Cham.

\end{thebibliography}
\end{document}